\pdfoutput=1
\documentclass[reqno,oneside,9pt]{amsart}
\DeclareFontEncoding{OT2}{}{}
\DeclareFontSubstitution{OT2}{cmr}{m}{n}
\DeclareFontFamily{OT2}{cmr}{\hyphenchar\font45 }
\DeclareFontShape{OT2}{cmr}{m}{n}{
<5><6><7><8><9>gen*wncyr
<10><10.95><12><14.4><17.28><20.74><24.88>wncyr10}{}
\DeclareFontShape{OT2}{cmr}{b}{n}{
<5><6><7><8><9>gen*wncyb
<10><10.95><12><14.4><17.28><20.74><24.88>wncyb10}{}
\DeclareMathAlphabet{\mathcyr}{OT2}{cmr}{m}{n}
\DeclareMathAlphabet{\mathcyb}{OT2}{cmr}{b}{n}
\SetMathAlphabet{\mathcyr}{bold}{OT2}{cmr}{b}{n}
\usepackage{geometry}
\usepackage{fancybox,ascmac,comment}
\usepackage{amssymb}
\usepackage{amsfonts}
\usepackage{amsmath}
\usepackage{amsthm}
\usepackage{mathtools}
\usepackage{enumerate}

\usepackage{mleftright}
\mleftright
\usepackage[all]{xy}
\mathtoolsset{showonlyrefs=true}
\numberwithin{equation}{section}
\makeatletter
\newcommand{\shortmathcal}[1]{\@tfor\ch:=#1\do{
\expandafter\edef\csname c\ch\endcsname{\noexpand\mathcal{\ch}}
}}
\newcommand{\shortmathbb}[1]{\@tfor\ch:=#1\do{
\expandafter\edef\csname bb\ch\endcsname{\noexpand\mathbb{\ch}}
}}
\newcommand{\shortmathbf}[1]{\@tfor\ch:=#1\do{
\expandafter\edef\csname b\ch\endcsname{\noexpand\mathbf{\ch}}
}}
\newcommand{\shortboldsymbol}[1]{\@tfor\ch:=#1\do{
\expandafter\edef\csname bs\ch\endcsname{\noexpand\boldsymbol{\ch}}
}}
\newcommand{\shortmathfrak}[1]{\@tfor\ch:=#1\do{
\expandafter\edef\csname f\ch\endcsname{\noexpand\mathfrak{\ch}}}}
\newcommand{\shortmathscr}[1]{\@tfor\ch:=#1\do{
\expandafter\edef\csname s\ch\endcsname{\noexpand\mathscr{\ch}}}}
\newcommand{\shortmathrm}[1]{\@tfor\ch:=#1\do{
\expandafter\edef\csname r\ch\endcsname{\noexpand\mathrm{\ch}}
}}
\makeatother
\shortmathcal{ABCDEFGHIJKLMNOPQRSTUVWXYZ}
\shortmathbb{ABCDEFGHIJKLMNOPQRSTUVWXYZ}
\shortmathbf{ABCDEFGHIJKLMNOPQRSTUVWXYZabcdefghijklmnopqrstuvwxyz}
\shortboldsymbol{ABCDEFGHIJKLMNOPQRSTUVWXYZabcdefghijklmnopqrstuvwxyz}
\shortmathfrak{ABCDEFGHIJKLMNOPQRSTUVWXYZabcdefghjklmnopqrstuvwxyz}
\shortmathscr{ABCDEFGHIJKLMNOPQRSTUVWXYZ}
\shortmathrm{ABCDEFGHIJKLMNOPQRSTUVWXYZabcdefghijklmnopqrstuvwxyz}
\newcommand{\sh}{\mathbin{\mathcyr{sh}}}
\newcommand{\emp}{\varnothing}
\newcommand{\dch}{\mathrm{dch}}

\newcommand{\pr}{\mathrm{pr}}
\newcommand{\vv}[2]{\begin{pmatrix}#1\\ #2\end{pmatrix}}
\newcommand{\Li}{\mathrm{Li}}
\renewcommand{\Re}{\mathrm{Re}}
\renewcommand{\Im}{\mathrm{Im}}
\DeclareMathOperator*{\Reg}{Reg}
\DeclareMathOperator{\reg}{reg}
\theoremstyle{definition}
\newtheorem{theorem}{Theorem}[section]
\newtheorem{proposition}[theorem]{Proposition}
\newtheorem{lemma}[theorem]{Lemma}
\newtheorem{corollary}[theorem]{Corollary}

\newtheorem{definition}[theorem]{Definition}
\theoremstyle{remark}
\newtheorem{remark}[theorem]{Remark}
\newtheorem{example}[theorem]{Example}
\allowdisplaybreaks
\title{Discrete iterated integrals and cyclic sum formulas}
\author{Hanamichi Kawamura}
\address[Hanamichi Kawamura]{Department of Mathematics, Faculty of Science Division I, Tokyo University of Science, 1-3 Kagurazaka, Shinjuku-ku, Tokyo, 162-8601, Japan}
\email{1121026@ed.tus.ac.jp}
\allowdisplaybreaks
\begin{document}
\begin{abstract}
    In this paper, we consider a discrete version of iterated integrals by the naive (equally divided) Riemann sum.
    In particular, basic three formulas for usual iterated integrals are discretized.
    Moreover, we proved cyclic sum formulas for discrete iterated integrals.
    They imply the cyclic sum formula for multiple polylogarithms. 
\end{abstract}
\maketitle
\section{Introduction}
We often encounter \emph{iterated integrals} like
\[I_{\dch}(1_{0};e_{a_{1}}\cdots e_{a_{k}};(-1)_{1})=\int_{0<t_{1}<\cdots<t_{k}<1}\frac{dt_{1}}{t_{1}-a_{1}}\cdots\frac{dt_{k}}{t_{k}-a_{k}}\qquad (a_{1},\ldots,a_{k}\in\bbC\setminus(0,1)),\]
in recent researches of number theory, knot theory, mathematical physics, etc.
One of most important examples of iterated integrals is the \emph{multiple zeta value}, which has the expression as a nested sum other than the iterated integral:
\begin{equation}\label{eq:mzv_ii}
(-1)^{r}I_{\dch}(1_{0};e_{1}e_{0}^{k_{1}-1}\cdots e_{1}e_{0}^{k_{r}-1};(-1)_{1})=\zeta(k_{1},\ldots,k_{r})\coloneqq\sum_{0<n_{1}<\cdots<n_{r}}\frac{1}{n_{1}^{k_{1}}\cdots n_{r}^{k_{r}}}.
\end{equation}
General iterated integrals, denoted as $I_{\gamma}(x;w;y)$, satisfy many important formulas.
The typical examples are the following.
\begin{theorem}\label{thm:basic_ii}
    Let $\gamma$ and $\gamma'$ be paths.
    Denote by $x_{\circ}$ (resp.~$y_{\circ}$) the initial (resp.~final) point of $\circ\in\{\gamma,\gamma'\}$.
    Assume that $y_{\gamma}$ is equal to $x_{\gamma'}$.
    Then, for $a_{1},\ldots,a_{k}\in D$ and $w,w'\in\fH_{D}$, we have 
    \begin{enumerate}
        \item(Path composition formula) $\displaystyle I_{\gamma\gamma'}(x_{\gamma};e_{a_{1}}\cdots e_{a_{k}};y_{\gamma'})=\sum_{i=0}^{k}I_{\gamma}(x_{\gamma};e_{a_{1}}\cdots e_{a_{i}};y_{\gamma})I_{\gamma'}(x_{\gamma'};e_{a_{i+1}}\cdots e_{a_{k}};y_{\gamma'})$.
        \item(Reversal formula) $I_{\gamma}(x_{\gamma};e_{a_{1}}\cdots e_{a_{k}};y_{\gamma})=(-1)^{k}I_{\gamma^{-1}}(y_{\gamma};e_{a_{k}}\cdots e_{a_{1}};x_{\gamma})$.
        \item(Shuffle product formula) $I_{\gamma}(x_{\gamma};w\sh w';y_{\gamma})=I_{\gamma}(x_{\gamma};w;y_{\gamma})I_{\gamma}(x_{\gamma};w';y_{\gamma})$.
    \end{enumerate} 
\end{theorem}
Recently, Maesaka, Seki and Watanabe \cite{msw24} discovered that:
\begin{theorem}[{\cite[Theorem 1.3]{msw24}}]
    Let $N,r,k_{1},\ldots,k_{r}$ be positive integers.
    Put $k\coloneqq k_{1}+\cdots+k_{r}$ and define $S\coloneqq\{0,k_{1},k_{1}+k_{2},\ldots,k_{1}+\cdots+k_{r}\}$.
    Then we have
    \[\sum_{\substack{0=n_{0}\le n_{1}\le\cdots\le n_{k}\le n_{k+1}=N\\ i\in S\implies n_{i}\neq n_{i+1}}}\left(\prod_{i\in S\setminus\{k\}}\frac{1}{N-n_{i+1}}\right)\left(\prod_{i\in\{0,\ldots,k-1\}\setminus S}\frac{1}{n_{i+1}}\right)=\sum_{0<n_{1}<\cdots<n_{r}<N}\frac{1}{n_{1}^{k_{1}}\cdots n_{r}^{k_{r}}}.\]
\end{theorem}
Comparing this fact and \eqref{eq:mzv_ii}, we can consider the left-hand side above as a \emph{discretization} of $I_{\dch}(1_{0};e_{1}e_{0}^{k_{1}-1}\cdots e_{1}e_{0}^{k_{r}-1};(-1)_{1})$
In this paper, we consider a discretization $\Delta_{N,\gamma}(s(\gamma);w;t(\gamma))$ for more general iterated integrals and associated notions.
The basic three formulas for iterated integrals, which are known as the \emph{path composition formula}, \emph{reversal formula} and \emph{shuffle product formula} respectively, are discretized:
\begin{theorem}
    Let $\gamma$ and $\gamma'$ be composable discrete paths.
    Then, for $a_{1},\ldots,a_{k}\in D$ and a $\gamma$-shuffleable pair $(w,w')\in(\fH_{D})^{2}$, we have 
    \begin{enumerate}
        \item(Path composition formula, Corollary \ref{cor:path_composition}) \[\Delta_{N,\gamma\gamma'}(s(\gamma);e_{a_{1}}\cdots e_{a_{k}};t(\gamma'))=\sum_{i=0}^{k}\Delta_{N,\gamma}(s(\gamma);e_{a_{1}}\cdots e_{a_{i}};t(\gamma))\Delta_{N,\gamma'}(s(\gamma');e_{a_{i+1}}\cdots e_{a_{k}};t(\gamma')).\]
        \item(Reversal formula, Theorem \ref{thm:reversal}) \[\Delta_{N,\gamma}(s(\gamma);e_{a_{1}}\cdots e_{a_{k}};t(\gamma))=(-1)^{k}\Delta_{N,\gamma^{-1}}(t(\gamma);e_{a_{k}}\cdots e_{a_{1}};s(\gamma)).\]
        \item(Asymptotic shuffle product formula, Theorem \ref{thm:spf}) \[\Delta_{N,\gamma}(s(\gamma);w;t(\gamma))\Delta_{N,\gamma}(s(\gamma);w';t(\gamma))=\Delta_{N,\gamma}(s(\gamma);w\sh w';t(\gamma))+O(N^{-1}\log^{\bullet}N),\qquad\text{as }N\to\infty.\]
    \end{enumerate} 
\end{theorem}
In the second main theorem, we evaluate a cyclic sum on discrete iterated integrals along a straight path.
\begin{theorem}[$=$ Theorem \ref{thm:csf}]
    Let $w=e_{a_{1}}\cdots e_{a_{k}}$ ($k\ge 2$) be an element of $\fH_{D}$ and $\gamma=(x,y)$ be a positive straight piece.
    Then we have
    \begin{multline}
        \sum_{\substack{1\le i\le k\\ a_{i}\neq x}}\Delta_{N,\gamma}((x,+);e_{a_{i}}\cdots e_{a_{k}}e_{a_{1}}\cdots e_{a_{i-1}}e_{x};(y,-))+\sum_{\substack{1\le i\le k\\ a_{i}\neq y}}\Delta_{N,\gamma}((x,+);e_{y}e_{a_{i+1}}\cdots e_{a_{k}}e_{a_{1}}\cdots e_{a_{i-1}}e_{x})\\
        =\sum_{\substack{1\le i\le k\\ a_{i}\in\{x,y\}}}\Delta_{N,\gamma}((x,+);e_{a_{i}}\cdots e_{a_{k}}e_{a_{1}}\cdots e_{a_{i}};(y,-))-\sum_{i=1}^{k}\Delta_{N,\gamma}((x,+);e_{a_{1},\ldots,a_{i},a_{i},\ldots,a_{k}};(y,-)).
\end{multline}
\end{theorem}
As its consequence, we show that there is a generalization of the cyclic sum formula for multiple zeta values (\cite{ho03}) for multiple polylogarithms (see Corollary \ref{cor:csf_mpl}).\\
This paper is organized as follows: we give a brief review about usual iterated integrals in Section \ref{sec:ii}.
After introducing discrete iterated integrals and proving basic properties in Section \ref{sec:dii}, we see how discrete iterated integrals converge to usual ones in Section \ref{sec:recovering}.
Finally, in Section \ref{sec:csf}, we prove the cyclic sum formula for discrete iterated integrals along a straight path.
\section*{Acknowledgements}
The author gratefully thanks to Prof.~Minoru Hirose, Prof.~Toshiki Matsusaka and Prof.~Shin-ichiro Seki for kindly letting him know about their result \cite{hms24}. 
\section{Review on usual iterated integrals}\label{sec:ii}
In what follows, we fix a finite subset $D$ of $\bbC$ and put $M\coloneqq\bbC\setminus D$ unless noted otherwise.
\subsection{Tangential base points and paths}
For $p\in D$ and a non-zero tangential vector $v\in T_{p}(\bbC)\simeq\bbC$, we call a tuple $v_{p}\coloneqq (p,v)$ a \emph{tangential base point}, where the isomorphism is induced by the canonical coordinate of $\bbC$.
Let $\pr(v_{p})\coloneqq p$ for a tangential base point $v_{p}$ and $\pr(p)\coloneqq p$ for a usual base point ($=$ an element of $\bbC\setminus D$) $p$. 
\begin{definition}[Path]
    Let $x$ and $y$ be (tangential or not) base points.
    We call $\gamma$ a path if $\gamma\colon[0,1]\to\bbC$ is a piecewise smooth path satisfying $\gamma(0)=\pr(x)$, $\gamma(1)=\pr(y)$, $\gamma((0,1))\in M$,
    \[\lim_{t\to +0}\frac{\gamma(t)-\pr(x)}{t}=v\qquad \text{if }x=v_{\pr(x)}\text{ is tangential}\]
    and
    \[\lim_{t\to 1-0}\frac{\gamma(t)-\pr(y)}{t-1}=-w\qquad \text{if }y=w_{\pr(y)}\text{ is tangential}.\]
    The point $x$ (resp.~$y$) is called the \emph{initial point} (resp.~the \emph{final point}) of $\gamma$.
\end{definition}
The set $\pi_{1}(M;y,x)$ of homotopy classes for such a path $\gamma$ (with the initial point $x$ and the final point $y$) is defined and they admit the composition
\[\pi_{1}(M;y,x)\times\pi_{1}(M;z,y)\to\pi_{1}(M;z,x);~(\gamma,\gamma')\mapsto\gamma\gamma'\]
and inversion $\pi_{1}(M;y,x)\to\pi_{1}(M;x,y);~\gamma\mapsto\gamma^{-1}$ (cf.~\cite[\S3.8.3]{bf24}).
For a path $\gamma$ and real numbers $a,b$ satisfying $0\le a<b\le 1$, we put $\gamma_{a,b}(t)\coloneqq\gamma((1-t)a+tb)$.
Note that, for a path $\gamma$, choosing suffciently small positive numbers $z$ and $z'$, we get the path $\gamma_{z,1-z'}$ whose initial and final point are not tangential.
When $(0,1)\cap D=\emp$, the path defined by $t\mapsto t$ is usually denoted by $\dch$.
\subsection{Algebraic setup}
Define $\fH_{D}$ be the non-commutative free $\bbQ$-algebra generated by formal symbols $e_{z}$, where $z$ runs over $D$.
For a path $\gamma$ with the initial point $x$ and the final point $y$, the subalgebra $\fH_{D}^{\gamma}$ of $\fH_{D}$ is defined by
\[\fH_{D}^{\gamma}\coloneqq\bbQ\oplus\bigoplus_{a\in D\setminus\{\pr(x),\pr(y)\}}\bbQ e_{a}\oplus\bigoplus_{\substack{a\in D\setminus\{\pr(x)\}\\ b\in D\setminus\{\pr(y)\}}}e_{a}\fH_{D}e_{b}.\]
\begin{definition}[Shuffle product]
    Define a $\bbQ$-bilinear product $\sh$ on $\fH_{D}$ by $1\sh w=w\sh 1=w$ for all $w\in\fH_{D}$ and the recursive rule
    \[e_{a}w\sh e_{b}w'=e_{a}(w\sh e_{b}w')+e_{b}(e_{a}w\sh w')\qquad (a,b\in D,~w,w'\in\fH_{D}).\]
    This product induces an associative and commutative $\bbQ$-algebra structure on $\fH_{D}$ and $\fH_{D}^{\gamma}$ for any path $\gamma$.  
\end{definition}
\begin{proposition}\label{prop:hopf}
    With the following structures, $\fH_{D}$ is a Hopf $\bbQ$-algebra.
    \begin{enumerate}
        \item The product $\fH_{D}\otimes\fH_{D}\to\fH_{D}$ is the shuffle product.
        \item The unit $\bbQ\to\fH_{D}$ is determined by $1\mapsto 1$.\\
        \item The coproduct $\fH_{D}\to\fH_{D}\otimes\fH_{D}$ is determined by $\displaystyle 
        \Delta^{\sh}(e_{a_{1}}\cdots e_{a_{k}})\coloneqq\sum_{i=0}^{k}e_{a_{1}}\cdots e_{a_{i}}\otimes e_{a_{i+1}}\cdots e_{a_{k}}$ ($a_{j}\in D$).
        \item The counit $\fH_{D}\to\bbQ$ is determined by $e_{a}\mapsto 0$ ($a\in D$).
        \item The antipode $\fH_{D}\to\fH_{D}$ is determined by $S^{\sh}(e_{a_{1}}\cdots e_{a_{k}})\coloneqq(-1)^{k}e_{a_{k}}\cdots e_{a_{1}}$. 
    \end{enumerate}
\end{proposition}
\subsection{Iterated integrals and regularization}
\begin{definition}[Convergent iterated integrals]\label{def:conv_ii}
    Let $\gamma$ be a path with the initial point $x$ and the final point $y$.
    For $w=e_{a_{1}}\cdots e_{a_{k}}\in\fH_{D}$, we put
    \[I_{\gamma}(x;w;y)\coloneqq\int_{0<t_{1}<\cdots<t_{k}<1}\prod_{i=1}^{k}\omega_{a_{i}}(\gamma(t_{i})),\]
    where $\omega_{a}(t)$ is the $1$-form $dt/(t-a)$.
    Extending it $\bbQ$-linearly, we define the $\bbQ$-linear map $I_{\gamma}(x;-;y)\colon\fH_{D}^{\gamma}\to\bbC$.
    We call this map (or its image) the (convergent) \emph{iterated integral}.
\end{definition}
\begin{definition}[Regularized iterated integrals]\label{def:reg_ii}
    Take a sufficiently small number $0<z<1$.
    Then, for any $w\in\fH_{D}$, there uniquely exists $P(x)\in\bbC[x]$ such that
    \[I_{\gamma_{z,1-z}}(\gamma(z);w;\gamma(1-z))=P(\log z)+O(z\log^{J}z)\qquad (z\to +0)\]
    for some $J>0$ (cf.~\cite[Lemma 3.351]{bf24}).
    From these data, we define the \emph{regularized iterated integral} by $I_{\gamma}(x;w;y)\coloneqq P(0)$.
\end{definition}
\begin{remark}
    On the above setting, if $w$ is an element of $\fH_{D}^{\gamma}$ the regularized iterated integral $I_{\gamma}(x;w;y)$ coincides with the iterated integral in the sense of Definition \ref{def:conv_ii}.
\end{remark}
The following three formulas are basic properties of iterated integrals.
\begin{theorem}[$=$ Theorem \ref{thm:basic_ii}]
    Let $\gamma$ and $\gamma'$ be paths.
    Denote by $x_{\circ}$ (resp.~$y_{\circ}$) the initial (resp.~final) point of $\circ\in\{\gamma,\gamma'\}$.
    Assume that $y_{\gamma}$ is equal to $x_{\gamma'}$.
    Then, for $a_{1},\ldots,a_{k}\in D$ and $w,w'\in\fH_{D}$, we have 
    \begin{enumerate}
        \item(Path composition formula) $\displaystyle I_{\gamma\gamma'}(x_{\gamma};e_{a_{1}}\cdots e_{a_{k}};y_{\gamma'})=\sum_{i=0}^{k}I_{\gamma}(x_{\gamma};e_{a_{1}}\cdots e_{a_{i}};y_{\gamma})I_{\gamma'}(x_{\gamma'};e_{a_{i+1}}\cdots e_{a_{k}};y_{\gamma'})$.
        \item(Reversal formula) $I_{\gamma}(x_{\gamma};e_{a_{1}}\cdots e_{a_{k}};y_{\gamma})=(-1)^{k}I_{\gamma^{-1}}(y_{\gamma};e_{a_{k}}\cdots e_{a_{1}};x_{\gamma})$.
        \item(Shuffle product formula) $I_{\gamma}(x_{\gamma};w\sh w';y_{\gamma})=I_{\gamma}(x_{\gamma};w;y_{\gamma})I_{\gamma}(x_{\gamma};w';y_{\gamma})$.
    \end{enumerate} 
\end{theorem}
\section{Discrete iterated integrals}\label{sec:dii}
In what follows, we fix a large positive even integer $N$ and put $\bbZ_{/N}\coloneqq \{n/N\mid n\in\bbZ\}\subseteq\bbQ$.
The symbol $D$ stands for a finite subset of $\bbC$.
\subsection{Discrete paths}
In this subsection, we discretize the theory of continuous paths connecting (tangential or not) base points.
\begin{definition}[Discrete paths]
    \begin{enumerate}
        \item The set of \emph{endpoints} is $B\coloneqq(\bbZ_{/N}\setminus D)\sqcup((\bbZ_{/N}\cap D)\times\{1,\ldots,N\})$.
        If $a\in D$, we also write $(a,+)\coloneqq(a,N/2)$ and $(a,-)\coloneqq (a,0)\coloneqq (a,N)$.
        \item Define a map $B\to\bbZ_{/N}$; $x\mapsto\hat{x}$ as $\hat{x}\coloneqq x$ if $x\in \bbZ_{/N}\setminus D$ and $\hat{x}\coloneqq a$ if $x=(a,i)\in(\bbZ_{/N}\cap D)\times\{1,\ldots,N\}$.
        \item We define the following three types of \emph{pieces}:
        \begin{enumerate}
            \item A \emph{positive straight piece} is an element $(x,y)\in(\bbZ_{/N})^{2}$ such that $x<y$ and the open interval $(x,y)$ does not intersect with $D$.
            Its \emph{initial point} is $(x,+)$ if $x\in D$ and is $x$ if $x\notin D$.
            Its \emph{final point} is $(y,-)$ if $y\in D$ and is $y$ if $y\notin D$.
            \item A \emph{negative straight piece} is an element $(x,y)\in(\bbZ_{/N})^{2}$ such that $y<x$ and the open interval $(-x,-y)$ does not intersect with $D$.
            Its initial point is $(x,-)$ if $x\in D$ and is $x$ if $x\notin D$.
            Its final point is $(y,+)$ if $y\in D$ and is $y$ if $y\notin D$.
            \item A \emph{positive singular piece} is an element $(x,a,b,+)$ of $(D\cap\bbZ_{/N})\times\{0,\ldots,N\}^{2}\times\{+,-\}$ such that $a<b$.
            Its initial point is $(x,a)$ and final point is $(x,b)$.
            \item A \emph{negative singular piece} is an element $(x,a,b,-)$ of $(D\cap\bbZ_{/N})\times\{0,\ldots,N\}^{2}\times\{+,-\}$ such that $a<b$.
            Its initial point is $(x,N-b)$ and final point is $(x,N-a)$.
            \item A \emph{trivial piece} is an element $x$ of copy $B'$ of $B$.
            Its initial and final points are $x$.
            If $p\in D$, we use the notation $(p,a,a,+)\coloneqq (p,a,a,-)\coloneqq (p,a)\in B'$ for every $a\in\{1,\ldots,N\}$ and regard it as one of singular pieces.
        \end{enumerate}
        \item Define maps $s$ (resp.~$t$) from the set of pieces to $B$ to be what assigns a piece to its initial (resp.~final) point.
        \item We say that a finite tuple $\gamma=(\gamma_{1},\ldots,\gamma_{r})$ ($r\ge 1$) is a \emph{discrete path} (on $\bbZ_{/N}\setminus D$) if each $\gamma_{i}$ is one of pieces above such that $t(\gamma_{i})=s(\gamma_{i+1})$ for every $i\in\{1,\ldots,k-1\}$.
        Define $s(\gamma)\coloneqq s(\gamma_{1})$ and $t(\gamma)\coloneqq t(\gamma_{r})$.
    \end{enumerate}
\end{definition}
Note that the above definition of discrete paths depends on $D$.
However, since we fix $D$, we omit the information of $D$ and simply say as ``a discrete path $\gamma$''.
\begin{definition}[Operations on discrete paths]
    ${}$
    \begin{enumerate}
        \item For a discrete path $\gamma_{i}=(\gamma_{i,1},\ldots,\gamma_{i,r_{i}})$ ($i\in\{1,2\}$) satisfying $t(\gamma_{1})=s(\gamma_{2})$, the \emph{composition} $\gamma_{1}\gamma_{2}$ is the concatenation $(\gamma_{1,1},\ldots,\gamma_{1,r_{1}},\gamma_{2,1},\ldots,\gamma_{2,r_{2}})$ of them as tuples of pieces.
        \item For a discrete path $\gamma=(\gamma_{1},\ldots,\gamma_{r})$, the \emph{inverse} is a discrete path $\gamma^{-1}\coloneqq(\gamma_{r}^{-1},\ldots,\gamma_{1}^{-1})$, where the inverse of pieces are defined as:
        \begin{enumerate}
            \item If $\gamma_{i}=(x,y)$ is a (positive or negative) straight piece, put $\gamma_{i}^{-1}\coloneqq(y,x)$.
            \item If $\gamma_{i}=(x,a,b,\sigma)$ is a (positive or negative) singular piece, put $\gamma_{i}^{-1}\coloneqq (x,N-b,N-a,-\sigma)$.
            \item If $\gamma_{i}$ is a trivial piece, its inverse is itself.
        \end{enumerate} 
    \end{enumerate}
\end{definition}
\begin{remark}\label{rem:composition_of_pieces}
From the definition, it is clear that a discrete path $\gamma=(\gamma_{1},\ldots,\gamma_{r})$ is composition of its pieces $\gamma_{1}\cdots\gamma_{r}$ by identifying the piece $\gamma_{i}$ with the discrete path $(\gamma_{i})$.
Hereafter we sometimes use this identification.
\end{remark}
\begin{definition}[Provision of discrete paths]
    Let $\gamma=(\gamma_{1},\ldots,\gamma_{r})$ be a discrete path.
    We define the \emph{provision} of $\gamma$ as the set $S_{\gamma}\coloneqq\bigsqcup_{i=1}^{r}(S_{\gamma_{i}}\times\{i\})$, where 
    $S_{\gamma_{i}}$ is determined as follows:
    \begin{enumerate}
        \item If $\gamma_{i}=(a,b)$ is a positive straight piece, $S_{\gamma_{i}}\coloneqq\{aN+1,\ldots,bN-1\}\subseteq\bbZ$.
        \item If $\gamma_{i}=(a,b)$ is a negative straight piece, $S_{\gamma_{i}}\coloneqq\{-aN+1,\ldots,-bN,-1\}\subseteq\bbZ$.
        \item If $\gamma_{i}=(x,a,b,\sigma)$ is a (positive or negative) singular piece, $S_{\gamma_{i}}$ is a set $\{\theta_{a+1}^{\sigma},\ldots,\theta_{b}^{\sigma}\}$, which consists of formal symbols $\theta_{j}^{\pm}$ (the sign represents whether the piece is positive or negative, which does not mean the inverse).
        \item If $\gamma_{i}$ is trivial, set $S_{\gamma_{i}}\coloneqq\emp$.
    \end{enumerate}
    For an element $n\in S_{\gamma}$, denote by $\tilde{n}$ the first component of $n$ and define $n+a\coloneqq (\tilde{n}+a,i)\in\bbZ\times\{1,\ldots,r\}$ for $a\in\bbZ$ and $n\in S_{\gamma_{i}}\subseteq S_{\gamma}$ if $\gamma_{i}$ is straight.
    The projection on the second component $S_{\gamma}\to\{1,\ldots,r\}$ is denoted by $J$.
\end{definition}
Moreover, we introduce the order $<_{\gamma}$ for the provision of a path $\gamma=(\gamma_{1},\ldots,\gamma_{r})$ in the following way.
\begin{enumerate}
    \item $(m,i)<_{\gamma}(n,i+1)$ for every $i\in\{1,\ldots,r-1\}$, $(m,i)\in S_{\gamma_{i}}$ and $(n,i+1)\in S_{\gamma_{i+1}}$.
    \item $(m,i)<_{\gamma}(n,i)$ if $\gamma_{i}$ is a positive or negative piece and $m<n$.
    \item $(\theta_{j}^{\pm},i)<_{\gamma}(\theta_{j+1}^{\pm},i)$ for $j\in\{1,\ldots,N-1\}$ if $\gamma_{i}$ is a (positive or negative) singular piece.
\end{enumerate}
The condition that $A<_{\gamma}B$ or $A=B$ holds is denoted by $A\le_{\gamma}B$ as usual. 
\begin{definition}[Elimination of discrete paths]
Let $\gamma=(\gamma_{1},\ldots,\gamma_{r})$ be a discrete path and $m,n\in S_{\gamma}$ satisfying $m\le_{\gamma}n$.
Write $m=(\tilde{m},i)$ and $n=(\tilde{n},j)$.
When $i\neq j$, we define a new discrete path $\gamma_{m,n}$ by $(\gamma'_{i},\gamma_{i+1},\ldots,\gamma_{j-1},\gamma''_{j})$ where
\[\gamma'_{i}\coloneqq\begin{cases}
    ((\tilde{m}-1)/N,\widehat{t(\gamma_{i})}) & \text{if }\gamma_{i}\text{ is a positive straight path},\\
    (-(\tilde{m}-1)/N,\widehat{t(\gamma_{i})}) & \text{if }\gamma_{i}\text{ is a negative straight path},\\
    (x,h,b,\pm) & \text{if }\gamma_{i}=(x,a,b,\pm)\text{ is a (positive or negative) singular path and }\tilde{m}=\theta_{h}^{\pm},
\end{cases}\]
and
\[\gamma''_{j}\coloneqq\begin{cases}
    (\widehat{s(\gamma_{j})},(\tilde{n}+1)/N) & \text{if }\gamma_{j}\text{ is a positive straight path},\\
    (\widehat{s(\gamma_{j})},-(\tilde{n}+1)/N) & \text{if }\gamma_{j}\text{ is a negative straight path},\\
    (x,a,h,\pm) & \text{if }\gamma_{j}=(x,a,b,\pm)\text{ is a (positive or negative) singular path and }\tilde{n}=\theta_{h}^{\pm}.
\end{cases}\]
If $i=j$, we define
\[\gamma_{m,n}\coloneqq\begin{cases}
    ((\tilde{m}-1)/N,(\tilde{n}+1)/N) & \text{if }\gamma_{i}\text{ is a positive straight path},\\
    (-(\tilde{m}-1)/N,-(\tilde{n}+1)/N) & \text{if }\gamma_{i}\text{ is a negative straight path},\\
    (x,j,j',\pm) & \text{if }\gamma_{i}=(x,a,b,\pm)\text{ is a (positive or negative) singular path, }\tilde{m}=\theta_{j}^{\pm}\text{ and }\tilde{n}=\theta_{j'}^{\pm}.
\end{cases}\]
If $m=n$, the path $\gamma_{m,n}$ is defined to be a trivial piece $\tilde{m}$ if $\gamma_{i}$ ($=\gamma_{j}$) is straight and to be a trivial piece $(x,j)$ if $\gamma_{i}=\gamma_{j}=(x,a,b,\pm)$ is singular, where $j$ is the subscript of $\tilde{m}=\theta_{j}^{\pm}$.
\end{definition} 
Note that, for a discrete path $\gamma$ and $a,b,m,n\in S_{\gamma}$ satisfying $a\le_{\gamma}m\le_{\gamma}n\le_{\gamma}b$, there is an inclusion $S_{\gamma_{m,n}}\subseteq S_{\gamma_{a,b}}\subseteq S_{\gamma}$ and the iterated elimination $(\gamma_{a,b})_{m,n}$ is equal to $\gamma_{m,n}$.
Moreover we see that $S_{\gamma_{a,b}}=\{n\in S_{\gamma}\mid a<_{\gamma}n<_{\gamma}b\}$ holds.
\subsection{Definition of discrete iterated integrals}
\begin{definition}[Discrete differential forms]
    Let $\gamma=(\gamma_{1},\ldots,\gamma_{r})$ be a discrete path, $n=(\tilde{n},i)\in S_{\gamma}$ and $a\in D$.
    Then we define $\omega_{a}(n)\in\bbC[\theta]$ by
    \[\omega_{a}^{N}(n)\coloneqq\begin{cases}
        \displaystyle\frac{1}{\tilde{n}-aN} & \text{if }\gamma_{i}\text{ is a positive straight piece},\\
        \displaystyle\frac{1}{\tilde{n}+aN} & \text{if }\gamma_{i}\text{ is a negative straight piece},\\
        \displaystyle\frac{\theta}{N} & \text{if }\gamma_{i}\text{ is a positive singular piece whose first entry is }a,\\
        -\displaystyle\frac{\theta}{N} & \text{if }\gamma_{i}\text{ is a negative singular piece whose first entry is }a,\\
        0 & \text{otherwise.}
    \end{cases}\]
    Furthermore, for elements $m=(\tilde{m},i)$ and $n=(\tilde{n},j)$ of $S_{\gamma}$, we define
    \[\omega^{N}(m,n)\coloneqq\begin{cases}
        \displaystyle\frac{1}{\tilde{n}-\tilde{m}} & \text{if }\gamma_{i}\text{ and }\gamma_{j}\text{ are both positive or negative straight pieces and }\tilde{n}\neq\tilde{m},\\
        0 & \text{otherwise.}
    \end{cases}\]
\end{definition}
\begin{definition}[Discrete iterated integrals]
    Let $\gamma$ be a discrete path and $w=e_{a_{1}}\cdots e_{a_{k}}\in\fH_{D}$.
    Then we define the \emph{discrete iterated integral} by
    \[\Delta_{N,\gamma}(s(\gamma);w;t(\gamma))\coloneqq\sum_{\substack{n_{1},\ldots,n_{k}\in S_{\gamma}\\n_{1}\le_{\gamma}\cdots\le_{\gamma}n_{k}}}\prod_{i=1}^{k}\omega_{a_{i}}^{N}(n_{i}).\]
    We put $\Delta_{N,\gamma}(s(\gamma);1;t(\gamma))\coloneqq 1$.
    Moreover, by the $\bbQ$-linear extension, we define $\Delta_{N,\gamma}(s(\gamma);-;t(\gamma))\colon\fH_{D}^{\gamma}\to\bbC[\theta]$.
\end{definition}
\begin{example}
    Let $D=\{0,1\}$, $w=e_{1}e_{0}^{2}$ and $\gamma=(\gamma_{1})$ be a discrete path which consists of only the positive straight piece $\gamma_{1}=(0,1)$.
    Then $S_{\gamma}=\{(1,1),(2,1)\ldots,(N-1,1)\}$ and $(i,1)\le_{\gamma}(j,1)$ if and only if $i\le j$.
    Thus we have
    \[\Delta_{N,\gamma}((0,+);e_{1}e_{0}^{2};(1,-))=\sum_{0<n_{1}\le n_{2}\le n_{3}<N}\frac{1}{(n_{1}-N)n_{2}n_{3}}.\]
    Taking the limit $N\to\infty$ and using \cite[Theorem 1.3]{msw24}, we obtain $\lim_{N\to\infty}\Delta_{N,\gamma}((0,+);e_{1}e_{0}^{2};(1,-))=I_{\dch}(1_{0};e_{1}e_{0}^{2};(-1)_{1})$.
    More general examples explaining such a limit which tends to the usual iterated integral appears in \S\ref{sec:recovering}.
\end{example}
\subsection{Basic formulas for discrete iterated integrals}
\begin{proposition}\label{prop:singular}
    Let $\gamma=(x,a,b,\pm)$ is a singular piece and identify it with a discrete path $(\gamma)$.
    Then we have
    \[\Delta_{N,\gamma}(s(\gamma);e_{a_{1}}\cdots e_{a_{k}};t(\gamma))=\begin{cases}
        \displaystyle \frac{(b-a)_{k}}{k!}\left(\frac{\pm\theta}{N}\right)^{k}& \text{if }a_{1}=\cdots=a_{k}=x,\\
        0 & \text{otherwise}
    \end{cases}\]
    for $a_{1},\ldots,a_{k}\in D$.
    Here $(X)_{k}\coloneqq X(X+1)\cdots (X+k-1)$ stands for the Pochhammer symbol.
\end{proposition}
\begin{proof}
    We proceed the proof by induction on $k$.
    The case where $k=0$ is obvious.
    If the case for some $k$ holds, we have
    \begin{align}
        \Delta_{N,\gamma}(s(\gamma);e_{a_{1}}\cdots e_{a_{k+1}};t(\gamma))
        &=\sum_{a<n_{1}\le\cdots\le n_{k+1}\le b}\prod_{i=1}^{k+1}\omega_{a_{i}}^{N}(\theta_{n_{i}}^{\pm})\\
        &=\sum_{a\le n\le b}\Delta_{N,(x,a,n,\pm)}((x,a);e_{a_{1}}\cdots e_{a_{k}};(x,n))\omega_{a_{k+1}}(\theta_{n}^{\pm}).
    \end{align}
    If $\{a_{1},\ldots,a_{k+1}\}\neq\{x\}$ we obtain the desired result.
    Otherwise from this expression we have
    \begin{align}
        \Delta_{N,\gamma}(s(\gamma);e_{a_{1}}\cdots e_{a_{k+1}};t(\gamma))
        &=\sum_{a<n\le b}\frac{(n-a)_{k}}{k!}\left(\frac{\pm\theta}{N}\right)^{k}\cdot\frac{\pm\theta}{N}\\
        &=\left(\frac{\pm\theta}{N}\right)^{k+1}\sum_{a<n\le b}\frac{(n-a)_{k}}{k!}\\
        &=\left(\frac{\pm\theta}{N}\right)^{k+1}\sum_{a<n\le b}\left(\frac{(n-a)_{k+1}}{(k+1)!}-\frac{(n-a-1)_{k+1}}{(k+1)!}\right)\\
        &=\frac{(b-a)_{k+1}}{(k+1)!}\left(\frac{\pm\theta}{N}\right)^{k+1}.
    \end{align}
\end{proof}
The following two lemmas are obvious from the definition.
\begin{lemma}\label{lem:path_composition}
    Let $\gamma=(\gamma_{1},\ldots,\gamma_{r})$ be a discrete path and $w=e_{a_{1}}\cdots e_{a_{k}}\in\fH_{D}$.
    Then we have
    \[\Delta_{N,\gamma}(s(\gamma);e_{a_{1}}\cdots e_{a_{k}};t(\gamma))
    =\sum_{0=i_{0}\le i_{1}\le\cdots\le i_{r-1}\le i_{r}=k}\prod_{j=1}^{r}\sum_{\substack{n_{i_{j-1}+1},\ldots,n_{i_{j}}\in S_{\gamma_{j}}\\ n_{i_{j-1}+1}\le_{\gamma_{j}}\cdots\le_{\gamma_{j}}n_{i_{j}}}}\prod_{h=i_{j-1}+1}^{i_{j}}\omega_{a_{h}}^{N}(n_{h}).\]
\end{lemma}
\begin{corollary}[Path composition formula]\label{cor:path_composition}
    Let $\gamma$ and $\gamma'$ be composable discrete paths (i.e., $t(\gamma)=s(\gamma')$).
    Then, for $w=e_{a_{1}}\cdots e_{a_{k}}\in\fH_{D}$, we have 
    \[\Delta_{N,\gamma\gamma'}(s(\gamma);e_{a_{1}}\cdots e_{a_{k}};t(\gamma'))=\sum_{i=0}^{k}\Delta_{N,\gamma}(s(\gamma);e_{a_{1}}\cdots e_{a_{i}};t(\gamma))\Delta_{N,\gamma'}(s(\gamma');e_{a_{i+1}}\cdots e_{a_{k}};t(\gamma')).\]
\end{corollary}
\begin{theorem}[Reversal formula]\label{thm:reversal}
    Let $\gamma$ be a discrete path.
    Then, for $a_{1},\ldots,a_{k}\in D$, we have 
    \[\Delta_{N,\gamma}(s(\gamma);e_{a_{1}}\cdots e_{a_{k}};t(\gamma))=(-1)^{k}\Delta_{N,\gamma^{-1}}(t(\gamma);e_{a_{k}}\cdots e_{a_{1}};s(\gamma)).\]
\end{theorem}
\begin{proof}
    Let $\gamma=(\gamma_{1},\ldots,\gamma_{r})$.
    Applying Corollary \ref{cor:path_composition} repeatedly, we have
    \[\Delta_{N,\gamma}(s(\gamma);e_{a_{1}}\cdots e_{a_{k}};t(\gamma))=\sum_{0=i_{0}\le i_{1}\le\cdots\le i_{r-1}\le i_{r}=k}\prod_{j=1}^{r}\Delta_{N,\gamma_{j}}(s(\gamma_{j});e_{a_{i_{j-1}+1}}\cdots e_{a_{i_{j}}};t(\gamma_{j}))\]
    and
    \[\Delta_{N,\gamma^{-1}}(t(\gamma);e_{a_{k}}\cdots e_{a_{1}};s(\gamma))=\sum_{0=i'_{0}\le i'_{1}\le\cdots\le i'_{r-1}\le i'_{r}=k}\prod_{j=1}^{r}\Delta_{N,\gamma_{r+1-j}^{-1}}(s(\gamma_{r+1-j});e_{a_{i'_{r+1-j}}}\cdots e_{a_{i'_{r-j}+1}};s(\gamma_{r+1-j})).\]
    Since we can give a one-to-one correspondence $(i_{1},\ldots,i_{r-1})\mapsto (i'_{1},\ldots,i'_{r})\coloneqq(k-i_{r-1},\ldots,k-i_{1})$ between the ranges of each sum above, it is sufficient to prove
    \[\Delta_{N,\gamma_{j}}(s(\gamma_{j});e_{a_{i_{j-1}+1}}\cdots e_{a_{i_{j}}};t(\gamma_{j}))=(-1)^{i_{j}-i_{j-1}}\Delta_{N,\gamma_{j}^{-1}}(t(\gamma_{j});e_{a_{i_{j}}}\cdots e_{a_{i_{j-1}+1}};s(\gamma_{j}))\]
    for $(i_{1},\ldots,i_{r-1})$ satisfying $0=i_{0}\le i_{1}\le\cdots\le i_{r-1}\le i_{r}=k$ and $j\in\{1,\ldots,r\}$.
    If $\gamma_{j}$ is a (positive or negative) straight path, we obtain this identity by substitution $n\mapsto -n$.
    When $\gamma_{j}=(x,a,b,\pm)$ is a singular path, from Proposition \ref{prop:singular} we have
    \begin{align}
        &\Delta_{N,\gamma_{j}}(s(\gamma_{j});e_{a_{i_{j-1}+1}}\cdots e_{a_{i_{j}}};t(\gamma_{j}))\\
        &=\begin{cases}
            \displaystyle\frac{(b-a)_{i_{j}-i_{j-1}}}{(i_{j}-i_{j-1})!}\left(\frac{\pm\theta}{N}\right)^{i_{j}-i_{j-1}} & \text{if }a_{i_{j-1}+1}=\cdots=a_{i_{j}}=x,\\
            0 & \text{otherwise},
        \end{cases}\\
        &=(-1)^{i_{j}-i_{j-1}}\begin{cases}
            \displaystyle\frac{((N-a)-(N-b))_{i_{j}-i_{j-1}}}{(i_{j}-i_{j-1})!}\left(\frac{\mp\theta}{N}\right)^{i_{j}-i_{j-1}} & \text{if }a_{i_{j-1}+1}=\cdots=a_{i_{j}}=x,\\
            0 & \text{otherwise},
        \end{cases}\\
        &=(-1)^{i_{j}-i_{j-1}}\Delta_{N,\gamma_{j}^{-1}}(t(\gamma_{j});e_{a_{i_{j}}}\cdots e_{a_{i_{j-1}+1}};s(\gamma_{j})).
    \end{align}
\end{proof}
\begin{lemma}\label{lem:simplest_convergent}
    Let $\gamma=(x,y)$ be a positive straight piece (regarded as a discrete path) and $a\in D$.
    Then we have
    \[\sum_{xN\le n\le yN}\frac{1}{|n-aN|}=O(\log N),\qquad\text{as }N\to\infty.\]
\end{lemma}
\begin{proof}
    For the case where $a\notin\{x,y\}$, put
        \[\rho\coloneqq\begin{cases} \Re(a)-y & \text{if }\Re(a)>y,\\
            x-\Re(a) & \text{if }x>\Re(a),\\
            \Im(a) & \text{otherwise.}
        \end{cases}\] 
        Then this quantity does not depend on $N$ and $|n-aN|\ge\rho N>0$ always holds for an integer $xN<n<yN$.
        Thus we have
        \[\sum_{xN<n<yN}\frac{1}{|n-aN|}\le\frac{(y-x)N-1}{\rho N}.\]
    If $a=x$, we obtain
    \begin{align}
        \sum_{xN<n<yN}\frac{1}{|n-xN|}
        &=\sum_{0<n<(y-x)N}\frac{1}{n}\\
        &=\log N+O(1),\qquad\text{as }N\to\infty.
    \end{align}
    The estimation for the $a=y$ case is also obtained by a similar computation. 
\end{proof}
For treat more general convergences, we define an auxiliary algebra $\cP_{D}\coloneqq\bbQ\langle e_{a_{1},\ldots,a_{m}}\mid a_{i}\in D,~m\ge 1\rangle$ and auxiliary sums
\[\Delta_{N,\gamma}(s(\gamma);w;t(\gamma))\coloneqq\sum_{\substack{n_{1},\ldots,n_{k}\in S_{\gamma}\\ n_{1}\le_{\gamma}\cdots \le_{\gamma}n_{k}}}\prod_{i=1}^{k}\prod_{j=1}^{m_{i}}\omega_{a_{i,j}}(n_{i}),\]
for a straight piece $\gamma$ and $w=e_{a_{1,1},\ldots,a_{1,m_{1}}}\cdots e_{a_{k,1},\ldots,a_{k,m_{k}}}\in\cP_{D}$.
Under the same setting, we also define its variant as
\[|\Delta|_{N,\gamma}(s(\gamma);w;t(\gamma))\coloneqq\sum_{\substack{n_{1},\ldots,n_{k}\in S_{\gamma}\\ n_{1}\le_{\gamma}\cdots \le_{\gamma}n_{k}}}\prod_{i=1}^{k}\prod_{j=1}^{m_{i}}|\omega_{a_{i,j}}(n_{i})|.\]
\begin{lemma}\label{lem:convergence}
    Let $\gamma$ be a straight piece and $w=e_{a_{1,1},\ldots,a_{1,m_{1}}}\cdots e_{a_{k,1}}\cdots e_{a_{k,m_{k}}}\in\cP_{D}$.
    Put $x\coloneqq\widehat{s(\gamma)}$ and $y\coloneqq\widehat{t(\gamma)}$.
    Then we have the following:
    \begin{enumerate}
        \item As $N\to\infty$, the sums $\Delta_{N,\gamma}(s(\gamma);w;t(\gamma))$ and $|\Delta|_{N,\gamma}(s(\gamma);w;t(\gamma))$ are $O(\log^{\bullet}N)$.
        \item $\Delta_{N,\gamma}(s(\gamma);w;t(\gamma))=O(N^{-1}\log^{\bullet}N)$ as $N\to\infty$ if there exists $1\le i\le k$ and $1\le j<j'\le m_{i}$ such that $m_{i}\ge 2$ and $a_{i,j}\neq a_{i,j'}$.
        \item $\Delta_{N,\gamma}(s(\gamma);w;t(\gamma))=O(N^{-1}\log^{\bullet}N)$ as $N\to\infty$ if there exists $1\le i\le k$ and $1\le j<j'\le m_{i}$ such that $m_{i}\ge 2$ and $a_{i,j}=a_{i,j'}\in D\setminus\{x,y\}$.
        \item $\Delta_{N,\gamma}(s(\gamma);w;t(\gamma))=O(N^{-1}\log^{\bullet}N)$ as $N\to\infty$ if there exists $1\le i<i'\le k$, $1\le j\le m_{i}$ and $1\le j'<j''\le m_{i'}$ such that $m_{i'}\ge 2$, $a_{i',j'}=a_{i',j''}=x$ and $a_{i,j}\neq x$.
        \item $\Delta_{N,\gamma}(s(\gamma);w;t(\gamma))=O(N^{-1}\log^{\bullet}N)$ as $N\to\infty$ if there exists $1\le i'<i\le k$, $1\le j\le m_{i}$ and $1\le j'<j''\le m_{i'}$ such that $m_{i'}\ge 2$, $a_{i',j'}=a_{i',j''}=y$ and $a_{i,j}\neq y$.
    \end{enumerate}
\end{lemma}
\begin{proof}
    Arguments similar to the below proofs for the positive straight case also works for the negative straight case.
    Thus we deal with only the positive straight case. 
    For $1\le i\le k$ and $1\le j\le m_{i}$ define
    \[
        \ba_{i}\coloneqq (a_{i,1},\ldots,a_{i,m_{i}})
        \]
        and
    \[\ba_{i}\langle j\rangle\coloneqq (a_{i,1},\ldots,a_{i,j-1},a_{i,j+1},\ldots,a_{i,m_{i}})\]
    \begin{enumerate}
        \item For the case where $\gamma$ is positive straight, it is sufficient that the latter case is proven by the triangle inequality.
        From Lemma \ref{lem:simplest_convergent} we have
        \[|\Delta|_{N,\gamma}(s(\gamma);w;t(\gamma))\le\sum_{\substack{xN<n_{i,j}<yN\\ 1\le i\le k\\ 1\le j\le m_{i}}}\prod_{j=1}^{m}\frac{1}{|n_{i,j}-a_{i,j}N|}=O(\log^{mk}N),\qquad\text{as }N\to\infty.\]
        \item By the partial fraction decomposition
        \[\frac{1}{n_{i}-a_{i,j}N}\frac{1}{n_{i}-a_{i,j'}N}=\frac{1}{N(a_{i,j}-a_{i,j'})}\left(\frac{1}{n_{i}-a_{i,j}N}-\frac{1}{n_{i}-a_{i,j'}N}\right),\]
        we obtain
        \begin{align}
            \Delta_{N,\gamma}(s(\gamma);w;t(\gamma))
            &=\sum_{xN<n_{1}\le\cdots\le n_{k}<yN}\prod_{h=1}^{k}\prod_{h'=1}^{m_{i}}\frac{1}{n_{h}-a_{h,h'}N}\\
            &=\begin{multlined}[t]
                \frac{1}{N(a_{i,j}-a_{i,j'})}\left(\Delta_{N,\gamma}(s(\gamma);e_{\ba_{1}}\cdots e_{\ba_{i-1}}e_{\ba_{i}\langle j\rangle}e_{\ba_{i+1}}\cdots e_{\ba_{k}};t(\gamma))\right.\\
                \left.-\Delta_{N,\gamma}(s(\gamma);e_{\ba_{1}}\cdots e_{\ba_{i-1}}e_{\ba_{i}\langle j'\rangle}e_{\ba_{i+1}}\cdots e_{\ba_{k}};t(\gamma))\right).
            \end{multlined}
        \end{align}
        Then the desired result follows from this equality and (1).
        If $\gamma$ is singular, 
        \item Put
        \[\rho\coloneqq\begin{cases} \Re(a_{i,j})-y & \text{if }\Re(a_{i,j})>y,\\
            x-\Re(a_{i,j}) & \text{if }x>\Re(a_{i,j}),\\
            \Im(a_{i,j}) & \text{otherwise.}
        \end{cases}\] 
        Then $\rho$ is independent of $N$ and the inequality $|n_{i}-a_{i,j}N|\ge\rho N>0$ always holds for an integer $xN<n_{i,j}<yN$.
        Thus we have
        \begin{align}
            |\Delta_{N,\gamma}(s(\gamma);w;t(\gamma))|
            &\le|\Delta|_{N,\gamma}(s(\gamma);w;t(\gamma))\\
            &\le\frac{1}{\rho N}|\Delta|_{N,\gamma}(s(\gamma);e_{\ba_{1}}\cdots e_{\ba_{i-1}}e_{\ba_{i}\langle j\rangle}e_{\ba_{i+1}}\cdots e_{\ba_{k}};t(\gamma))\\
            &\stackrel{(1)}{=}O(N^{-1}\log^{\bullet}N)\qquad\text{as }N\to\infty.
        \end{align}
        \item We have
        \begin{align}
            |\Delta|_{N,\gamma}(s(\gamma);w;t(\gamma))
            &=\sum_{xN<n_{1}\le\cdots\le n_{k}<yN}\prod_{\substack{1\le h\le k\\ 1\le h'\le m_{h}}}\frac{1}{|n_{h}-a_{h,h'}N|}\\
            &\le\sum_{xN<n_{1}\le\cdots\le n_{k}<yN}\frac{1}{|n_{i}-xN|}\prod_{\substack{1\le h\le k\\ 1\le h'\le m_{h}\\ (h,h')\neq (i',j')}}\frac{1}{|n_{h}-a_{h,h'}N|}\\
            &=|\Delta|_{N,\gamma}(s(\gamma);e_{\ba_{1}}\cdots e_{\ba_{i-1}}e_{\ba_{i},x}e_{\ba_{i+1}}\cdots e_{\ba_{i'-1}}e_{\ba_{i'}\langle j'\rangle}e_{\ba_{i'+1}}\cdots e_{\ba_{k}};t(\gamma))
    \end{align}
        since $|n_{i'}-xN|\ge |n_{i}-xN|$.
        Then we can use (2) for the most right-hand side above.
        \item We have
        \begin{align}
            |\Delta|_{N,\gamma}(s(\gamma);w;t(\gamma))
            &=\sum_{xN<n_{1}\le\cdots\le n_{k}<yN}\prod_{\substack{1\le h\le k\\ 1\le h'\le m_{h}}}\frac{1}{|n_{h}-a_{h,h'}N|}\\
            &\le\sum_{xN<n_{1}\le\cdots\le n_{k}<yN}\frac{1}{|n_{i}-yN|}\prod_{\substack{1\le h\le k\\ 1\le h'\le m_{h}\\ (h,h')\neq (i',j')}}\frac{1}{|n_{h}-a_{h,h'}N|}\\
            &=|\Delta|_{N,\gamma}(s(\gamma);e_{\ba_{1}}\cdots e_{\ba_{i'-1}}e_{\ba_{i'}\langle j'\rangle}e_{\ba_{i'+1}}\cdots e_{\ba_{i-1}}e_{\ba_{i},y}e_{\ba_{i+1}}\cdots e_{\ba_{k}};t(\gamma))
    \end{align}
        since $|n_{i'}-yN|\ge |n_{i}-yN|$.
        Then we can use (2) for the most right-hand side above.
    \end{enumerate}
\end{proof}
From here we use the $O$-notation for also elements of $\bbC[\theta]$, that is, we write as
\[\sum_{i=0}^{k}a_{i,N}\theta^{i}=O(b_{N}),\qquad\text{as }N\to\infty,\]
for $k\in\bbZ_{\ge 0}$ and sequences $a_{i,-}\colon\bbZ_{\ge 1}\to\bbC$ ($0\le i\le k$) satisfying $a_{i,N}=O(b_{N})$ as $N\to\infty$ for every $0\le i\le k$. 
\begin{remark}
    The estimations (2)--(5) appearing in Lemma \ref{lem:convergence} hold for also the case where $\gamma$ is singular.
    Indeed, for $\gamma=(x,a,b,\pm)$ we easily obtain
    \[\Delta_{N,\gamma}((s,\gamma);w;(t,\gamma))=\begin{cases}
    \displaystyle\frac{(b-a)_{k}}{k!}\left(\frac{\pm\theta}{N}\right)^{mk} & \text{if every }a_{i,j}\text{ is }x,\\
    0 & \text{otherwise.}
    \end{cases}\] 
\end{remark}
\begin{definition}
    Let $\gamma=(\gamma_{1},\ldots,\gamma_{r})$ be a discrete path and $w=e_{a_{1}}\cdots e_{a_{k}}$ and $w'=e_{b_{1}}\cdots e_{b_{l}}$ elements of $\fH_{D}$.
    Then we say that a pair $(w,w')\in(\fH_{D})^{2}$ is \emph{$\gamma$-shuffleable} if, for every increasing sequences $0=i_{0}\le i_{1}\le\cdots\le i_{r}=k$ and $0=j_{0}\le j_{1}\le\cdots\le j_{r}=l$ and $1\le c\le r$, either $a_{i_{c-1}+1}$ or $b_{j_{c-1}+1}$ (resp.~$a_{i_{c}}$ or $b_{j_{c}}$) does \emph{not} coincide with $\widehat{s(\gamma_{c})}$ (resp.~$\widehat{t(\gamma_{c})}$).
\end{definition}
We remark that, if $\gamma$ is a piece, the above condition asserts that $\{a_{1},b_{1}\}\neq\{\widehat{s(\gamma)}\}$ and $\{a_{k},b_{l}\}\neq\{\widehat{t(\gamma)}\}$.
\begin{theorem}[Asymptotic shuffle product formula]\label{thm:spf}
    Let $\gamma$ be a discrete path and $w=e_{a_{1}}\cdots e_{a_{k}}$ and $w'=e_{b_{1}}\cdots e_{b_{l}}$ elements of $\fH_{D}$.
    Assume that $(w,w')$ is $\gamma$-shuffleable.
    Then we have
    \begin{equation}\label{eq:spf}
        \Delta_{N,\gamma}(s(\gamma);w;t(\gamma))\Delta_{N,\gamma}(s(\gamma);w';t(\gamma))=\Delta_{N,\gamma}(s(\gamma);w\sh w';t(\gamma))+O(N^{-1}\log^{\bullet}N),\qquad\text{as }N\to\infty.
    \end{equation}
\end{theorem}
\begin{proof}
    The proof proceeds in a way similar to \cite[Proposition 2.4]{seki24}.
    First we consider the case where $\gamma$ is a piece.
    By definition we have
    \begin{align}
        &\Delta_{N,\gamma}(s(\gamma);w;t(\gamma))\Delta_{N,\gamma}(s(\gamma);w';t(\gamma))\\
        &=\sum_{\substack{m_{1},\ldots,m_{k},n_{1},\ldots,n_{l}\in S_{\gamma}\\ m_{1}\le_{\gamma}\cdots\le_{\gamma}m_{k}\\ n_{1}\le_{\gamma}\cdots\le_{\gamma}n_{l}}}\left(\prod_{i=1}^{k}\omega_{a_{i}}^{N}(m_{i})\right)\left(\prod_{j=1}^{l}\omega_{b_{j}}^{N}(n_{j})\right)\\
        &=\Delta_{N,\gamma}((s,\gamma);w\sh w';(t,\gamma))+\sum_{\substack{m_{1},\ldots,m_{k},n_{1},\ldots,n_{l}\in S_{\gamma}\\ m_{1}\le_{\gamma}\cdots\le_{\gamma}m_{k}\\ n_{1}\le_{\gamma}\cdots\le_{\gamma}n_{l}\\ \exists(i,j),~m_{i}=n_{j}}}\left(\prod_{i=1}^{k}\omega_{a_{i}}^{N}(m_{i})\right)\left(\prod_{j=1}^{l}\omega_{b_{j}}^{N}(n_{j})\right).
    \end{align}
    Focusing on the second term on the right-hand side, we see that $\Delta_{N,\gamma}(s(\gamma);w;t(\gamma))\Delta_{N,\gamma}(s(\gamma);w';t(\gamma))-\Delta_{N,\gamma}((s,\gamma);w\sh w';(t,\gamma))$ can be decomposed to a finite sum of $\Delta_{N,\gamma}$ values satisfying one of the conditions in (2)--(5) of Lemma \ref{lem:convergence}.
    Indeed, we can use (2) if $a_{i}\neq b_{j}$, (3) if $a_{i}=b_{j}\notin\{\widehat{s(\gamma)},\widehat{t(\gamma)}\}$, (4) and the assumption $a_{1}\neq\widehat{s(\gamma)}$ if $a_{i}=b_{j}=\widehat{s(\gamma)}$ or (5) and the assumption $b_{l}\neq\widehat{t(\gamma)}$ if $a_{i}=b_{j}=\widehat{t(\gamma)}$.
    Thus \eqref{eq:spf} holds if $\gamma$ is a piece.
    Next we deal with the general case.
    We use the temporary notation $f_{i}(X)\coloneqq\Delta_{N,\gamma_{i}}(s(\gamma_{i});X;t(\gamma_{i}))$ for $X\in\fH_{D}$.
    Using Corollary \ref{cor:path_composition}, we have
    \[\Delta_{N,\gamma}(s(\gamma);w;t(\gamma))\Delta_{N,\gamma}(s(\gamma);w';t(\gamma))
    =\sum_{\substack{0=i_{0}\le i_{1}\le\cdots\le i_{r}=k\\ 0=j_{0}\le j_{1}\le\cdots\le j_{r}=l}}\prod_{c=1}^{r}\left(f_{c}(e_{a_{i_{c-1}+1}}\cdots e_{a_{i_{c}}})f_{c}(e_{b_{j_{c-1}+1}}\cdots e_{b_{j_{c}}})\right).\]
    From the assumption that $(w,w')$ is $\gamma$-shuffleable and the piecewise result, we have 
    \[f_{c}(e_{a_{i_{c-1}+1}}\cdots e_{a_{i_{c}}})f_{c}(e_{b_{j_{c-1}+1}}\cdots e_{b_{j_{c}}})=f_{c}(e_{a_{i_{c-1}+1}}\cdots e_{a_{i_{c}}}\sh e_{b_{j_{c-1}+1}}\cdots e_{b_{j_{c}}})+O(N^{-1}\log^{\bullet}N),\qquad\text{as }N\to\infty.\]
    By this fact and Proposition \ref{prop:hopf}, we obtain
    \begin{align}
        &(\mu_{\bbC}\circ(f_{1}\otimes\cdots\otimes f_{r})\circ\underbrace{\Delta^{\sh}\circ\cdots\circ\Delta^{\sh}}_{r-1})(w\sh w')\\
        &=\sum_{\substack{w_{1}\cdots w_{r}=w\\ w'_{1}\cdots w'_{r}=w'}}(\mu_{\bbC}\circ(f_{1}\otimes\cdots\otimes f_{r}))((w_{1}\otimes\cdots\otimes w_{r})\sh (w'_{1}\otimes\cdots\otimes w'_{r}))\\
        &=\sum_{\substack{w_{1}\cdots w_{r}=w\\ w'_{1}\cdots w'_{r}=w'}}(\mu_{\bbC}\circ(f_{1}\otimes\cdots\otimes f_{r}))((w_{1}\sh w'_{1})\otimes\cdots\otimes (w_{r}\sh w'_{r}))\\
        &=\sum_{\substack{0=i_{0}\le i_{1}\le\cdots\le i_{r}=k\\ 0=j_{0}\le j_{1}\le\cdots\le j_{r}=l}}\prod_{c=1}^{r}\left(f_{c}(e_{a_{i_{c-1}+1}}\cdots e_{a_{i_{c}}}\sh e_{b_{j_{c-1}+1}}\cdots e_{b_{j_{c}}})\right)\\
        &=\Delta_{N,\gamma}(s(\gamma);w;t(\gamma))\Delta_{N,\gamma}(s(\gamma);w';t(\gamma))+O(N^{-1}\log^{\bullet}N),\qquad\text{as }N\to\infty,
    \end{align}
    where we use the symbol $\sh$ for the product on $(\fH_{D})^{\otimes r}$ induced from $\sh$ by abuse of notation and $\mu_{\bbC}$ for the product for $\bbC[\theta]$.
    The leftmost side is $\Delta_{N,\gamma}(s(\gamma);w\sh w';t(\gamma))$ by the path composition formula.
    This completes the proof.
\end{proof}
\section{Recovering formulas for usual iterated integrals}\label{sec:recovering}
In this section, we deal with the following types of paths:
\begin{definition}
    \begin{enumerate}
        \item Define a continuous path $\beta$ to be $\dch\cdot\alpha\cdot\dch^{-1}$, where $\alpha$ is the path which starts from $(-1)_{1}$, circles $1$ counterclockwise and ends on $(-1)_{1}$.
        \item Define a discrete path $\beta^{\Delta}=(\gamma_{1},\gamma_{2},\gamma_{3})$, where $\gamma_{1}=(0,1)$, $\gamma_{2}=(1,0,N,+)$ and $\gamma_{3}=(1,0)$. 
    \end{enumerate}
    Moreover, we also write the straight piece $(0,1)$ as $\dch^{\Delta}$.
\end{definition}
Hereafter we assume that $D$ contains $0$ and $1$ and that $D\setminus\{0\}\subseteq M'\coloneqq\{z\mid |z|\ge 1\}\subseteq\bbC$.
Put 
\[\cI'\coloneqq\left\{\vv{z_{1},\ldots,z_{r}}{k_{1},\ldots,k_{r}}\left|\begin{array}{c}r\ge 1,\\k_{1},\ldots,k_{r}\in\bbZ_{\ge 1},\\ z_{1},\ldots,z_{r}\in D,\\ (k_{r},z_{r})\neq (1,1)\end{array}\right.\right\}.\]
\subsection{Preliminaries: harmonic product}
Let $\fH^{1}\coloneqq\bbQ\oplus e_{1}\bbQ\langle e_{0},e_{1}\rangle$ and $\fH^{0}\coloneqq\bbQ\oplus e_{1}\bbQ\langle e_{0},e_{1}\rangle e_{0}$.
Endow $\fH^{1}$ with a product $\ast$ by the following rules: $1\ast w=w\ast 1=w$ for every $w\in\fH^{1}$ and
\[e_{a}w\ast e_{b}w'\coloneqq e_{ab}(w\ast e_{b}w'+e_{a}w\ast w'-e_{0}(w\ast w'))\]
for $a,b\in\{0,1\}$ and $w,w'\in\fH^{1}$.
This provides a $\bbQ$-algebra structure on $\fH^{1}$ and its subalgebra structure on $\fH^{0}$.
\begin{proposition}[Hoffman {\cite{hoffman97}}]\label{prop:hoffman}
    Every $w\in\fH^{1}$ has the expression
    \[w=\sum_{i=0}^{n}w_{i}\ast e_{1}^{i},\]
    by the unique non-negative integer $n$ and $w_{i}\in\fH^{0}$.
\end{proposition}
The constant term $w_{0}$ on the above proposition is denoted by $\reg_{\ast}(w)$.
Moreover we introduce a variant of harmonic regularization in \cite{ikz06}.
\begin{definition}
    Assume that a sequence $a\colon\bbZ_{\ge 1}\to\bbC$ has the asymptotic expansion of type
    \[a_{N}=P(\log N+\gamma)+O(N^{-1}\log^{\bullet}N),\qquad\text{as }N\to\infty,\]
    where $P(X)\in\bbC[X]$ and $\gamma$ denotes the Euler constant.
    Note that such an expansion is unique if its exists.
    Then we define
    \[\Reg_{N\to\infty}^{*}a_{N}\coloneqq P(0).\]
    Moreover, for $f\in(\bbC^{\bbZ_{\ge 1}})[\theta]$, we extend the definition as
    \[\Reg_{N\to\infty}^{\ast}f(N,\theta)\coloneqq\sum_{i=0}^{n}\left(\Reg_{N\to\infty}^{\ast}a_{N,i}\right)\theta^{i},\]
    where we write the coefficients as $f(N,\theta)=\sum_{i=0}^{n}a_{N,i}\theta^{i}$ and note that the degree $n$ does not depend on $N$. 
\end{definition}
\begin{remark}
    The operation $\Reg_{N\to\infty}^{\ast}$ is compatible with the addition and product.
    Furthermore, when $\{a_{n}\}_{n}$ converges, $\Reg_{N\to\infty}^{*}a_{N}=\lim_{N\to\infty}a_{N}$ holds.
\end{remark}
\begin{definition}
    For $w\in\fH_{D}$ and $\gamma\in\{\dch,\beta\}$, we define
    \[I^{\Delta}_{\gamma}(w)\coloneqq\Reg_{N\to\infty}^{\ast}\Delta_{N,\gamma^{\Delta}}(s(\gamma);w;t(\gamma)).\]
\end{definition}
\begin{lemma}\label{lem:harmonic_mzv}
    Let $k_{1},\ldots,k_{r}\ge 1$.
    Then we have
    \begin{equation}\label{eq:harmonic}
        I^{\Delta}_{\dch}(e_{1}e_{0}^{k_{1}-1}\cdots e_{1}e_{0}^{k_{r}-1})=(-1)^{r}\zeta^{\ast}(k_{1},\ldots,k_{r}),
    \end{equation}
    where $\zeta^{\ast}$ stands for the harmonic regularization of multiple zeta values:
    \[\zeta^{\ast}(k_{1},\ldots,k_{r})\coloneqq\Reg_{N\to\infty}^{*}\sum_{0<n_{1}<\cdots<n_{r}<N}\frac{1}{n_{1}^{k_{1}}\cdots n_{r}^{k_{r}}}.\]
\end{lemma}
\begin{proof}
    Combining \cite[Theorem 1.3]{msw24} and Lemma \ref{lem:convergence}, we have
    \[\Delta_{N,(0,1)}((0,+);u;(1,-))=Z_{<N}(u)+O(N^{-1}\log^{\bullet}N),\]
    for $u\in\fH^{1}$, where $Z_{<N}\colon\fH^{1}\to\bbQ$ is the $\bbQ$-linear map determined by $1\mapsto 1$ and
    \[e_{1}e_{0}^{l_{1}-1}\cdots e_{1}e_{0}^{l_{s}-1}\mapsto (-1)^{s}\sum_{0<n_{1}<\cdots<n_{s}<N}\frac{1}{n_{1}^{l_{1}}\cdots n_{s}^{l_{s}}},\qquad (l_{1},\ldots,l_{s}\ge 1).\]
    Therefore we obtain
    \begin{align}
        I^{\Delta}_{\dch}(e_{1}e_{0}^{k_{1}-1}\cdots e_{1}e_{0}^{k_{r}-1})
        &=\Reg_{N\to\infty}^{\ast}\Delta_{N,(0,1)}((0,+);w;(1,-))\\
        &=\Reg_{N\to\infty}^{\ast}Z_{<N}(w)\\
        &=(-1)^{r}\zeta^{\ast}(k_{1},\ldots,k_{r}),
    \end{align}
    where $w=e_{1}e_{0}^{k_{1}-1}\cdots e_{1}e_{0}^{k_{r}-1}$.
\end{proof}
\subsection{The straight case}
\begin{theorem}[Hirose--Matsusaka--Seki {\cite{hms24}}]\label{thm:hms}
    Let $\iota=\vv{z_{1},\ldots,z_{r}}{k_{1},\ldots,k_{r}}\in\cI'$.
    Then we have
    \[\lim_{N\to\infty}\Delta_{N,(0,1)}((0,+);W(\iota);(1,-))=I_{\dch_{0,1}}(1_{0};W(\iota);(-1)_{1}).\]
\end{theorem}
\begin{proof}
    Combine \cite[Theorem 1.2]{hms24} and \cite[Proposition 2.5]{hms24}.
\end{proof}
\begin{definition}
    The \emph{multiple polylogarithm} (of shuffle-type) is defined by
    \[\Li^{\sh}\vv{z_{1},\ldots,z_{r}}{k_{1},\ldots,k_{r}}\coloneqq\sum_{0<n_{1}<\cdots<n_{r}}\frac{z_{1}^{n_{1}}z_{2}^{n_{2}-n_{1}}\cdots z_{r}^{n_{r}-n_{r-1}}}{n_{1}^{k_{1}}\cdots n_{r}^{k_{r}}}\]
    for complex numbers $z_{1},\ldots,z_{r}$ and $k_{1},\ldots,k_{r}\in\bbZ_{\ge 1}$ satisfying $\vv{z_{1}^{-1},\ldots,z_{r}^{-1}}{k_{1},\ldots,k_{r}}\in\cI'$.
\end{definition}
For $\iota=\vv{z_{1},\ldots,z_{r}}{k_{1},\ldots,k_{r}}\in\cI'$, put
    \[W(\iota)\coloneqq e_{z_{1}}e_{0}^{k_{1}-1}\cdots e_{z_{r}}e_{0}^{k_{r}-1}.\]
    Then it holds that
    \[I_{\dch_{0,1}}(1_{0};W(\iota);(-1)_{1})=(-1)^{r}\Li^{\sh}\vv{1/z_{1},\ldots,1/z_{r}}{k_{1},\ldots,k_{r}}.\]
    Therefore, combining this and Theorem \ref{thm:hms} we obtain
    \begin{equation}\label{eq:mpl}
        \lim_{N\to\infty}\Delta_{N,(0,1)}((0,+);W(\iota);(1,-))=(-1)^{r}\Li^{\sh}\vv{1/z_{1},\ldots,1/z_{r}}{k_{1},\ldots,k_{r}}.
    \end{equation}
For the next subsection, we introduce the extension
\[\rL^{\sh}\vv{z_{1},\ldots,z_{r}}{k_{1},\ldots,k_{r}}\coloneqq I_{\dch}^{\Delta}(e_{1/z_{1}}e_{0}^{k_{1}-1}\cdots e_{1/z_{r}}e_{0}^{k_{r}-1}),\]
for $k_{1},\ldots,k_{r}\in\bbZ_{\ge 1}$ and $z_{1},\ldots,z_{r}\in\bbC$ satisfying $z_{i}^{-1}\in D$.
\subsection{Variants of symmetric multiple polylogarithms}
In this subsection, imitating the construction of $\zeta^{RS}$ in \cite{hirose20}, we show that the symmetric multiple polylogarithms and symmetric multiple zeta values are recovered from discrete iterated integrals.
For a polynomial $P(\theta)=a_{0}+\cdots+a_{n}\theta^{n}\in\bbC[\theta]$, we put $P(\theta)[\theta]\coloneqq a_{1}$.
\begin{theorem}\label{thm:smpl}
    Let $\alpha,k_{1},\ldots,k_{r}$ be positive integers and $z_{1},\ldots,z_{r+1}\in D$.
    Assume that $u/z_{i}\in D$ for every $1\le i\le r+1$ and $u\in D$.
    Then we have
    \begin{multline}
        \sum_{u\in D}u^{\alpha}I^{\Delta}_{\beta}(e_{u/z_{1}}e_{0}^{k_{1}-1}\cdots e_{u/z_{r}}e_{0}^{k_{r}-1}e_{u/z_{r+1}})[\theta]\\
        =\sum_{i=0}^{r}(-1)^{k_{i+1}+\cdots+k_{r}}z_{i+1}^{\alpha}\rL^{\sh}\vv{z_{1}/z_{i+1},\ldots,z_{i}/z_{i+1}}{k_{1},\ldots,k_{i}}\rL^{\sh}\vv{z_{r+1}/z_{i+1},\ldots,z_{i+2}/z_{i+1}}{k_{r},\ldots,k_{i+1}}.
    \end{multline}
\end{theorem}
\begin{proof}
    Set $k\coloneqq k_{1}+\cdots+k_{r}$ and
    \[w_{u}\coloneqq e_{a_{1}}\cdots e_{a_{k+1}}\coloneqq e_{u/z_{1}}e_{0}^{k_{1}-1}\cdots e_{u/z_{r}}e_{0}^{k_{r}-1}e_{u/z_{r+1}}\]
    for $u\in D$.
    By the (discrete) path composition formula (Corollary \ref{cor:path_composition}), for each $u\in D$ we have
    \begin{multline}
        \Delta_{N,\beta^{\Delta}}((0,+);w_{u};(+,0))\\
        =\sum_{0\le i\le j\le k+1}\Delta_{N,\dch^{\Delta}}((0,+);e_{a_{1}}\cdots e_{a_{i}};(1,-))\Delta_{N,\alpha^{\Delta}}((1,-);e_{a_{i+1}}\cdots e_{a_{j}};(1,-))\Delta_{N,(\dch^{\Delta})^{-1}}((1,-);e_{a_{j+1}}\cdots e_{a_{k}};(0,+)).
    \end{multline}
    Since Proposition \ref{prop:singular} guarantees that $\Delta_{N,\alpha^{\Delta}}((1,-);e_{a_{i+1}}\cdots e_{a_{j}};(1,-))$ only if $a_{i+1}=\cdots=a_{j}=1$, we obtain
    \begin{multline}
        \Delta_{N,\beta^{\Delta}}((0,+);w_{u};(+,0))\\
        =\sum_{\substack{0\le i\le j\le r\\ k_{i+1}=\cdots=k_{j}=1\\ z_{i+1}=\cdots=z_{j+1}=u}}\Delta_{N,\dch^{\Delta}}((0,+);e_{u/z_{1}}e_{0}^{k_{1}-1}\cdots e_{u/z_{i}}e_{0}^{k_{i}-1};(1,-))\frac{(N)_{j-i+1}}{(j-i+1)!}\left(\frac{\theta}{N}\right)^{j-i+1}\\
        \cdot\Delta_{N,(\dch^{\Delta})^{-1}}((1,-);e_{0}^{k_{j+1}-1}e_{u/z_{j+2}}\cdots e_{0}^{k_{r}-1}e_{u/z_{r+1}};(0,+)).
    \end{multline}
    Applying $\Reg_{N\to\infty}^{\ast}\sum_{u\in D}$ after taking the coefficients at $\theta$ on both sides, we have
    \begin{align}
        \sum_{u\in D}u^{\alpha}I^{\Delta}_{\beta}(w_{u})[\theta]
    &=\begin{multlined}[t]
        \Reg_{N\to\infty}^{\ast}\sum_{u\in D}u^{\alpha}\sum_{\substack{0\le i=j\le r\\ k_{i+1}=\cdots=k_{j}=1\\ z_{i+1}=\cdots=z_{j+1}=u}}\Delta_{N,\dch^{\Delta}}((0,+);e_{u/z_{1}}e_{0}^{k_{1}-1}\cdots e_{u/z_{i}}e_{0}^{k_{i}-1};(1,-))\\
        \cdot\Delta_{N,(\dch^{\Delta})^{-1}}((1,-);e_{0}^{k_{j+1}-1}e_{u/z_{j+2}}\cdots e_{0}^{k_{r}-1}e_{u/z_{r+1}};(0,+))\end{multlined}\\
    &=\begin{multlined}[t]
        \sum_{0\le i\le r}z_{i+1}^{\alpha}\Reg_{N\to\infty}^{\ast}\left(\Delta_{N,\dch^{\Delta}}((0,+);e_{z_{i+1}/z_{1}}e_{0}^{k_{1}-1}\cdots e_{z_{i+1}/z_{i}}e_{0}^{k_{i}-1};(1,-))\right.\\
        \cdot\left.\Delta_{N,(\dch^{\Delta})^{-1}}((1,-);e_{0}^{k_{j+1}-1}e_{z_{i+1}/z_{i+2}}\cdots e_{0}^{k_{r}-1}e_{z_{i+1}/z_{r+1}};(0,+))\right)
    \end{multlined}\\
    &=\begin{multlined}[t]
        \sum_{0\le i\le r}(-1)^{k_{i+1}+\cdots+k_{r}}z_{i+1}^{\alpha}\Reg_{N\to\infty}^{\ast}\left(((0,+);e_{z_{i+1}/z_{1}}e_{0}^{k_{1}-1}\cdots e_{z_{i+1}/z_{i}}e_{0}^{k_{i}-1};(1,-))\right.\\
        \cdot\left.\Delta_{N,\dch^{\Delta}}((0,+);e_{z_{i+1}/z_{r+1}}e_{0}^{k_{r}-1}\cdots e_{z_{i+1}/z_{i+2}}e_{0}^{k_{i+1}-1};(1,-))\right)\end{multlined}\\
    &=\sum_{0\le i\le r}(-1)^{k_{i+1}+\cdots+k_{r}}z_{i+1}^{\alpha}\rL^{\sh}\vv{z_{1}/z_{i+1},\ldots,z_{i}/z_{i+1}}{k_{1},\ldots,k_{i}}\rL^{\sh}\vv{z_{r+1}/z_{i+1},\ldots,z_{i+2}/z_{i+1}}{k_{r},\ldots,k_{i+1}}.
\end{align}
\end{proof}
\begin{remark}
    In addition to the assumption in Theorem \ref{thm:smpl}, we assume that $z_{i}\neq z_{i+1}$ for $1\le i\le r$ satisfying $k_{i}=1$.
    Then we see that the right-hand side of Theorem \ref{thm:smpl} coincides with the symmetric multiple polylogarithm $\pounds_{\cS,\alpha}\vv{z_{1},\ldots,z_{r}}{k_{1},\ldots,k_{r}}$ (for definition, take modulo $t$ in \cite[Definition 4.15]{kawamura22}. See also \cite[Definition B.1.1]{jarossay19}).
\end{remark}
\begin{corollary}
    Let $k_{1},\ldots,k_{r}$ be positive integers and $D=\{0,1\}$.
    Then we have
    \[I_{\beta}^{\Delta}(e_{1}e_{0}^{k_{1}-1}\cdots e_{1}e_{0}^{k_{r}-1}e_{1})[\theta]=\zeta_{\cS}^{\ast}(k_{1},\ldots,k_{r})\coloneqq\sum_{i=0}^{r}(-1)^{k_{i+1}+\cdots+k_{r}}\zeta^{\ast}(k_{1},\ldots,k_{i})\zeta^{\ast}(k_{r},\ldots,k_{i+1}).\]
\end{corollary}
\begin{proof}
    It follows from Theorem \ref{thm:smpl} and Lemma \ref{lem:harmonic_mzv}.
\end{proof}
\section{Cyclic sum formulas for discrete iterated integrals}\label{sec:csf}
In this section, we prove the cyclic sum formulas for discrete iterated integrals along a straight piece.
Throughout this section, the symbol $\gamma$ always stands for a positive straight piece $(x,y)$ such that $x,y\in D$.
\begin{definition}[Hoffman--Ohno type sum]
    Let $w=e_{a_{1}}\cdots e_{a_{k}}$ be an element of $\fH_{D}$.
    We define
    \[\Delta_{N,\gamma}^{O}((x,+);w;(y,-))\coloneqq\sum_{\substack{n_{1},\ldots,n_{k+1}\in S_{\gamma}\\ n_{1}\le_{\gamma}\cdots\le_{\gamma}n_{k+1}}}\left(\prod_{i=1}^{k}\omega_{a_{i}}^{N}(n_{i})\right)\left(\omega^{N}(n_{1},n_{k+1})-\omega^{N}_{x}(n_{k+1})\right)\]
    and
    \[\Delta_{N,\gamma}^{O'}((x,+);w;(y,-))\coloneqq\sum_{\substack{n_{0},\ldots,n_{k}\in S_{\gamma}\\ n_{0}\le_{\gamma}\cdots\le_{\gamma}n_{k}}}\left(\prod_{i=1}^{k}\omega_{a_{i}}^{N}(n_{i})\right)\left(\omega_{y}^{N}(n_{0})-\omega^{N}(n_{k},n_{0})\right)\]
\end{definition}
\begin{proposition}[Transport relations]\label{prop:transport}
    Let $a_{1},\ldots,a_{k},z$ ($k\ge 1$) be elements of $D$ and put $w=e_{a_{1}}\cdots e_{a_{k}}$.
    Then we have
    \begin{enumerate}
        \item $\displaystyle\Delta_{N,\gamma}^{O}((x,+);we_{z};(y,-))=\Delta_{N,\gamma}^{O'}((x,+);we_{z};(y,-)),$
        \item \begin{multline}
            \Delta_{N,\gamma}^{O'}((x,+);we_{z};(y,-))=\Delta_{N,\gamma}^{O}((x,+);e_{z}w;(y,-))+(1-\delta_{x,z})\Delta_{N,\gamma}((x,+);e_{z}we_{x};(y,-))\\
            +(1-\delta_{y,z})\Delta_{N,\gamma}((x,+);e_{y}we_{z};(y,-))-(1-\delta_{x,z})(1-\delta_{y,z})\Delta_{N,\gamma}((x,+);e_{z}we_{z};(y,-))+\Delta_{N,\gamma}((x,+);e_{z,a_{1},\ldots,a_{k},z};(y,-)).
        \end{multline}
    \end{enumerate}
\end{proposition}
\begin{proof}
    ${}$
    \begin{enumerate}
        \item By definition we have
    \begin{align}
        &\Delta_{N,\gamma}^{O}((x,+);we_{z};(y,-))\\
        &=\sum_{\substack{n_{1},n_{k+1},n_{k+2}\in S_{\gamma}\\ n_{1}\le_{\gamma}n_{k+1}\le_{\gamma}n_{k+2}}}\omega_{a_{1}}(n_{1})\Delta_{N,\gamma_{n_{1},n_{k}}}(s(\gamma_{n_{1},n_{k}});e_{a_{2}}\cdots e_{a_{k}};t(\gamma_{n_{1},n_{k}}))\omega_{z}(n_{k})\left(\omega^{N}(n_{1},n_{k+1})-\omega_{x}^{N}(n_{k+1})\right).
    \end{align}
    Let us compute the sum about $n_{k+1}$.
    For $m,n\in S_{\gamma}$ satisfying $m\le_{\gamma}n$, define
    \[A(m,n)\coloneqq\sum_{\substack{c\in S_{\gamma}\\ n\le_{\gamma}c}}\left(\omega^{N}(m,c)-\omega_{x}^{N}(c)\right)\]
    and
    \[B(m,n)\coloneqq\sum_{\substack{d\in S_{\gamma}\\ d\le_{\gamma}m}}\left(\omega_{y}^{N}(d)-\omega^{N}(n,d)\right).\]
    It suffices to prove $A(m,n)=B(m,n)$.
    If $m\neq n$, by making the sum telescopic, we have
    \begin{align}
        A(m,n)
        &=\sum_{\tilde{n}\le c<yN}\left(\frac{1}{c-\tilde{m}}-\frac{1}{c-xN}\right)\\
        &=\sum_{\tilde{n}\le c<yN}\sum_{xN<d\le\tilde{m}}\left(\frac{1}{c-d}-\frac{1}{c-(d-1)}\right)\\
        &=\sum_{xN<d\le\tilde{m}}\left(\frac{1}{\tilde{n}-d}-\frac{1}{yN-d}\right)\\
        &=B(m,n).
    \end{align}
    The $m=n$ case is also proved by a similar argument.
    \item In the definition
    \[\Delta_{N,\gamma}^{O'}((x,+);we_{z};(y,-))=\sum_{\substack{n_{0},n_{k+1}\in S_{\gamma}\\ n_{0}\le_{\gamma}n_{k+1}}}\Delta_{N,\gamma_{n_{0},n_{k+1}}}(s(\gamma_{n_{0},n_{k+1}});w;t(\gamma_{n_{0},n_{k+1}}))\omega_{z}^{N}(n_{k+1})\left(\omega_{y}^{N}(n_{0})-\omega^{N}(n_{k+1},n_{0})\right),\]
    we can apply the decomposition
    \begin{align}
        &\omega_{z}^{N}(n_{k+1})\left(\omega_{y}^{N}(n_{0})-\omega^{N}(n_{k+1},n_{0})\right)\\
        &=\frac{1}{\tilde{n_{k+1}}-zN}\left(\frac{1}{\tilde{n_{0}}-yN}-\frac{1}{\tilde{n_{0}}-\tilde{n_{k+1}}}\right)\\
        &=\frac{1}{\tilde{n_{0}}-zN}\left(\frac{1}{\tilde{n_{k+1}}-\tilde{n_{0}}}-\frac{1}{\tilde{n_{k+1}}-xN}\right)+\frac{1-\delta_{x,z}}{\tilde{n_{0}}-zN}\frac{1}{\tilde{n_{k+1}}-xN}+\frac{1-\delta_{y,z}}{\tilde{n_{0}}-yN}\frac{1}{\tilde{n_{k+1}}-zN}-\frac{1-\delta_{x,z}}{\tilde{n_{0}}-zN}\frac{1-\delta_{y,z}}{\tilde{n_{k+1}}-zN}\\
        &=\begin{multlined}[t]
            \omega_{z}^{N}(n_{0})(\omega^{N}(n_{0},n_{k+1})-\omega^{N}_{x}(n_{k+1}))+(1-\delta_{x,z})\omega^{N}_{z}(n_{0})\omega^{N}_{x}(n_{k+1})\\
            +(1-\delta_{y,z})\omega^{N}_{y}(n_{0})\omega^{N}_{z}(n_{k+1})-(1-\delta_{x,z})(1-\delta_{y,z})\omega^{N}_{z}(n_{0})\omega^{N}_{z}(n_{k+1}),
        \end{multlined}
    \end{align}
    when $n_{0}\neq n_{k+1}$.
    If $n_{0}=n_{k+1}$, such a composition becomes the form
    \begin{align}
        &\omega_{z}^{N}(n_{k+1})\left(\omega_{y}^{N}(n_{0})-\omega^{N}(n_{k+1},n_{0})\right)\\
        &=\omega_{z}^{N}(n_{0})(\omega^{N}(n_{0},n_{k+1})-\omega^{N}_{x}(n_{k+1}))+(1-\delta_{x,z})\omega^{N}_{z}(n_{0})\omega^{N}_{x}(n_{k+1})\\
        +(1-\delta_{y,z})\omega^{N}_{y}(n_{0})\omega^{N}_{z}(n_{k+1}).
    \end{align}
    Therefore we obtain the result by summing up them.
\end{enumerate}
\end{proof}
\begin{theorem}[Cyclic sum formula]\label{thm:csf}
    Let $w=e_{a_{1}}\cdots e_{a_{k}}$ ($k\ge 2$) be an element of $\fH_{D}$.
    Then we have
    \begin{multline}
        \sum_{\substack{1\le i\le k\\ a_{i}\neq x}}\Delta_{N,\gamma}((x,+);e_{a_{i}}\cdots e_{a_{k}}e_{a_{1}}\cdots e_{a_{i-1}}e_{x};(y,-))+\sum_{\substack{1\le i\le k\\ a_{i}\neq y}}\Delta_{N,\gamma}((x,+);e_{y}e_{a_{i+1}}\cdots e_{a_{k}}e_{a_{1}}\cdots e_{a_{i-1}}e_{x})\\
        =\sum_{\substack{1\le i\le k\\ a_{i}\in\{x,y\}}}\Delta_{N,\gamma}((x,+);e_{a_{i}}\cdots e_{a_{k}}e_{a_{1}}\cdots e_{a_{i}};(y,-))-\sum_{i=1}^{k}\Delta_{N,\gamma}((x,+);e_{a_{1},\ldots,a_{i},a_{i},\ldots,a_{k}};(y,-)).
\end{multline}
\end{theorem}
\begin{proof}
    Using two assertinons in Proposition \ref{prop:transport} and putting $w=e_{a_{1}}\cdots e_{a_{k-1}}$ and $z=a_{k}$, we have
    \begin{multline}
        \Delta_{N,\gamma}^{O}((x,+);e_{a_{1}}\cdots e_{a_{k}};(y,-))
        =\Delta_{N,\gamma}^{O}((x,+);e_{a_{k}}e_{a_{1}}\cdots e_{a_{k-1}};(y,-))+(1-\delta_{x,a_{k}})\Delta_{N,\gamma}((x,+);e_{a_{k}}e_{a_{1}}\cdots e_{a_{k-1}}e_{x};(y,-))\\
            +(1-\delta_{y,a_{k}})\Delta_{N,\gamma}((x,+);e_{y}e_{a_{1}}\cdots e_{a_{k}};(y,-))-(1-\delta_{x,a_{k}})(1-\delta_{y,a_{k}})\Delta_{N,\gamma}((x,+);e_{a_{k}}e_{a_{1}}\cdots e_{a_{k}};(y,-))\\
            +\Delta_{N,\gamma}((x,+);e_{a_{1},\ldots,a_{i},a_{i},\ldots,a_{k}};(y,-)),
    \end{multline}
    Applying this equation $k$ times derives that the $\Delta_{N,\gamma}^{O}$ term vanishes and thus obtain the desired formula. 
\end{proof}
Making the setting suitable for multiple polylogarithms in the above theorem, we obtain the cyclic sum formulas for them as below.
This is a generalization of that for multiple $L$-values (\cite{ktw11}).
\begin{corollary}[Cyclic sum formula for multiple polylogarithms]\label{cor:csf_mpl}
    Let $k_{1},\ldots,k_{r}$ ($r\ge 1$) be positive integers and $z_{1},\ldots,z_{r}$ complex numbers $|z_{i}|\le 1$ for all $1\le i\le r$.
    Suppose that $\{z_{1},\ldots,z_{r}\}\neq 1$ if $k_{1}=\cdots =k_{r}=1$.
    Then we have
    \begin{multline}
        \sum_{i=1}^{r}\Li^{\sh}\vv{z_{i+1},\ldots,z_{r},z_{1},\ldots,z_{i}}{k_{i+1},\ldots,k_{r},k_{1},\ldots,k_{i-1},k_{i}+1}
        =\sum_{i=1}^{r}\sum_{j=0}^{k_{i}-2}\Li^{\sh}\vv{1,z_{i+1},\ldots,z_{r},z_{1},\ldots,z_{i}}{j+1,k_{i+1},\ldots,k_{r},k_{1},\ldots,k_{i-1},k_{i}-j}\\
        +\sum_{\substack{1\le i\le r\\ z_{i}\neq 1}}\Li^{\sh}\vv{1,z_{i+1},\ldots,z_{r},z_{1},\ldots,z_{i}}{k_{i+1},\ldots,k_{r},k_{1},\ldots,k_{i},1}
        -\sum_{\substack{1\le i\le r\\ z_{i}\neq 1}}\Li^{\sh}\vv{z_{i},\ldots,z_{r},z_{1},\ldots,z_{i}}{k_{i+1},\ldots,k_{r},k_{1},\ldots,k_{i},1}.
        \end{multline}
\end{corollary}
\begin{proof}
    Put $e_{a_{1}}\cdots e_{a_{k}}=e_{1/z_{1}}e_{0}^{k_{1}-1}\cdots e_{1/z_{r}}e_{0}^{k_{r}-1}$ in Theorem \ref{thm:csf}.
    The assumption guarantees that $\{a_{1},\ldots,a_{k}\}\neq 1$ and then we see that
    \[\lim_{N\to\infty}\Delta_{N,\gamma}((x,+);e_{a_{1},\ldots,a_{i},a_{i},\ldots,a_{k}};(y,-))=0\]
    by Lemma \ref{lem:convergence}.
    Thus we have the theorem by using \eqref{eq:mpl}.
\end{proof}
\section{Expected future works}
\begin{enumerate}
    \item Using notion of the discrete pullback (cf.~\cite[Definition 3.5.2]{hirani03}), it is expected that iterated integrals along more general paths are well discretized.
    The cyclic sum formulas for more general paths (for example, that for symmetric multiple zeta values \cite{hmo21}) are as well.
    \item In the definition of discrete iterated integrals, permitting $e_{z}$'s for $(z,i)\in S_{\gamma}$, one may obtain new (asymptotic) relations.
    In particular, we expect that, by extending the proof of Theorem \ref{thm:csf}, we may have the derivation relation for multiple zeta values (\cite{ikz06}) and its generalization to multiple polylogarithms.
    \item There are many formulas other than what treats cyclic sums, which is valid for general iterated integrals or hyperlogarithms, such as the confluence relation (\cite[Proposition 6.9]{hs20}), block shuffle identity (\cite[Theorem 42]{hs22}), etc.
    Do them have any good discretization?
\end{enumerate}


\begin{thebibliography}{MSW}
    \bibitem{bf24} J.~I.~Burgos Gil and J.~Fres\'{a}n, \emph{Multiple zeta values: from numbers to motives}, to appear in Clay Math.~Proc.
    \bibitem{jarossay19} D.~Jarossay, \emph{Adjoint cyclotomic multiple zeta values and cyclotomic multiple harmonic values}, preprint, \texttt{arXiv:1412.5099v4}.
    \bibitem{hirani03} A.~N.~Hirani, \emph{Discrete exterior calculus}, doctoral thesis in California Institute of Technology, 2003.
    \bibitem{hirose20} M.~Hirose, \emph{Double shuffle relations for refined symmetric multiple zeta values}, Doc.~Math.~\textbf{25} (2020), 365--380.
    \bibitem{hms24} M.~Hirose, T.~Matsusaka and S.~Seki, \emph{A discretization of the iterated integral expression
    of the multiple polylogarithm}, preprint, \texttt{arXiv:2404.15210}.
    \bibitem{hs20} M.~Hirose and S.~Sato, \emph{The motivic Galois group of mixed Tate motives over $\bbZ[1/2]$ and its action on the fundamental group of $\bbP^{1}\setminus\{0,\pm1,\infty\}$}, preprint, \texttt{arXiv.2007.04288}.
    \bibitem{hs22} M.~Hirose and S.~Sato, \emph{Block shuffle identities for multiple zeta values}, preprint, \texttt{arXiv:2206.03458}.
    \bibitem{hmo21} M.~Hirose, H.~Murahara and M.~Ono, \emph{On variants of symmetric multiple zeta-star values and the cyclic sum formula}, Ramanujan J.~(2021).
    \bibitem{hoffman97} M.~E.~Hoffman, \emph{The algebra of multiple harmonic series}, J.~Algebra \textbf{194} (1997), 477--495.
    \bibitem{ho03} M.~E.~Hoffman and Y.~Ohno, \emph{Relations of multiple zeta values and their algebraic expression}, J.~Algebra \textbf{262} (2003), 332--347.
    \bibitem{goncharov01} A.~B.~Goncharov, \emph{Multiple polylogarithms and mixed Tate motives}, preprint, \texttt{arXiv:math/0103059}.
    \bibitem{ikz06} K.~Ihara, M.~Kaneko and D.~Zagier, \emph{Derivation and double shuffle relations for multiple zeta values}, Compos.~Math.~\textbf{142} (2006), 307--338.
    \bibitem{kawamura22} H.~Kawamura, \emph{Iterated integrals for associated with colored rooted tree}, preprint, \texttt{arXiv:2209.13293}.
    \bibitem{ktw11} G.~Kawashima, T.~Tanaka and N.~Wakabayashi, \emph{Cyclic sum formula for multiple $L$-values}, J.~Algebra \textbf{348} (2011), 336--349.
    \bibitem{msw24} T.~Maesaka, S.~Seki and T.~Watanabe, \emph{Deriving two dualities simultaneously from a family of identities for multiple harmonic sums}, preprint, \texttt{arXiv:2402.05730}.
    \bibitem{seki24} S.~Seki, \emph{A proof of the extended double shuffle relation without using integrals}, preprint, \texttt{arXiv:2402.18300}.
\end{thebibliography}
\end{document}